\documentclass[11pt]{article}

\usepackage[margin = 1.0in]{geometry}   
\usepackage{amssymb}
\usepackage{amsmath}
\usepackage{palatino}
\usepackage{graphicx}
\usepackage[english]{babel}
\usepackage{amscd}
\usepackage{amsthm}
\usepackage{color}
\usepackage{enumerate}
\usepackage{hyperref}

\usepackage{blkarray}

\usepackage[mathscr]{euscript}
\newtheorem{theorem}{Theorem}[section]
\newtheorem{conjecture}{Conjecture}
\newtheorem{corollary}{Corollary}[section]
\newtheorem{lemma}{Lemma}[section]
\newtheorem{proposition}{Proposition}[section]
\newtheorem{definition}{Definition}

\newcommand{\Prb}{\mathbb{P}}
\newcommand{\Bin}{\textit{Bin}}
\newcommand{\Hyp}{\textit{Hyp}}
\newcommand{\mG}{\mathbb{G}}
\newcommand{\Exp}{\mathbb{E}}
\newcommand{\mH}{\mathcal{H}}
\newcommand{\mbH}{\mathbb{H}}
\newcommand{\mB}{\mathcal{B}}

\newcommand{\mD}{\mathcal{D}}
\newcommand{\mP}{\mathcal{P}}
\newcommand{\mF}{\mathcal{F}}

\title{A binomial random multigraph}
\author{
Christos Pelekis\thanks{Aristotle University of Thessaloniki, 
Department of Mathematics, 
541 24 Thessaloniki, Greece, e-mail: pelekis.chr@gmail.com, \, cpelekis@math.auth.gr }
}

\begin{document}

\maketitle

\begin{abstract}  
Fix a positive integer $n$, a real number $p\in (0,1]$, and a (perhaps random)
hypergraph $\mH$ on $[n]$. 
We introduce and investigate the following  
random multigraph model, which we denote  $\mG(n,p\, ; \,\mH)$: 
begin with an empty graph on $n$ vertices, which are labelled by the set $[n]$. 
For every $H\in \mH$ choose, independently from previous choices, a doubleton from $H$, say $D = \{i,j\} \subset H$, uniformly at random and then introduce an edge
between  the vertices $i$ and $j$ in the graph 
with probability $p$, where each edge is introduced independently of all other edges.   
\end{abstract}



\section{Introduction}

\subsection{Notation and terminology}

Let us begin with introducing some notation and terminology, which will remain fixed throughout the text. 

Given a positive integer $n$, we denote by 
$[n]$ the set $\{1,\ldots,n\}$. The cardinality of a finite set $F$ is denoted  $|F|$. 
A \emph{hypergraph}, $\mH$, on $[n]$ is a collection of subsets of $[n]$. Let us emphasise that 
we do not assume that each subset of $[n]$ appears at most once in $\mH$, and we allow repetitions.   More precisely, $\mH$ is a \emph{multi-hypergraph} but we slightly abuse terminology and use the term hypergraph instead.  The elements of a hypergraph are referred to as \emph{hyperedges}, or just \emph{edges}. 

Given a hypergraph $\mH$ on $[n]$ and a subset $H\subset [n]$, we denote by $m(H;\mH)$ the \emph{multiplicity} of $H$ in the hypergraph $\mH$, i.e., $m(H;\mH)$ is the number of times the set $H$ appears in the collection $\mH$. 
The \emph{degree} of a set $A\subset [n]$ in $\mH$, denoted $\deg_{\mH}(A)$, is the cardinality of the hypergraph $\{H\in \mH : A\subset H\}$. 
When $A=\{v\}$ is a singleton, we write $\deg_{\mH}(v)$ instead of $\deg_{\mH}(\{v\})$. 

A hypergraph $\mH$ is \emph{simple}
when $m(H;\mH) =1$, for all $H\in\mH$. 
%or, in other words, when $\mH\subset 2^{[n]}$, where $2^{[n]}$ denotes the collection consisting of all subsets of $[n]$.
Given a finite set $F\subset [n]$ and an integer $k\in \{0,1,\ldots, | F| \}$, we denote by $\binom{F}{k}$ the simple hypergraph consisting of all subsets of $F$ whose cardinality equals $k$.
A hypergraph $\mH$ is said to be $k$-\emph{uniform} if $|H|=k$, for all $H\in\mH$. 
A $2$-uniform hypergraph is just a (multi)graph. Unless mentioned otherwise, all hypergraphs $\mH$ considered in the text satisfy $|H|\ge 2$, for all $H\in\mH$. Finally, 
given a simple hypergraph, say $\mH$, and a positive integer $m$, we denote by $\mH^{(\times m)}$ the hypergraph that contains each hyperedge of $\mH$ exactly $m$ times.

Given two sequences $\{a_n\}_n, \{b_n\}_n$, we write $a_n\ll b_n$, or $b_n\gg a_n$, when $\frac{a_n}{b_n}\to 0$, as $n\to + \infty$. We also write $a_n = O(b_n)$ if there exists a constant $C>0$, which does not depend on $n$, and a positive integer $n_0$ such that $a_n \le C\cdot b_n$, for all $n\ge n_0$. Similarly, we write $a_n = \Theta(b_n)$ if there exist constants $C,c>0$, independent of $n$, and positive integer $n_0$ such that 
$cb_n \le a_n \le Cb_n$, for all $n\ge n_0$. In other words, $a_n$ and $b_n$ are of the same order. Finally, we write $a_n = \Omega(b_n)$ if  there exists constant $C>0$ such that $a_n \ge C b_n$ for all $n$ large enough. 

Given a positive integer $m$ and a real number $q\in [0,1]$, we denote by 
$\Bin(m,q)$ a \emph{binomial random variable} of parameters $m$ and $q$. 
Furthermore, given real numbers $\{p_s\}_{s\in S}\in [0,1]$, indexed by some finite set $S$,  
we denote by $\mB(p_s;s\in S)$ the random variable that counts the number of successes in $|S|$ independent trials, where trial $s\in S$ has probability $p_s$ of being a success. 
The random variable $\mB(p_s;s\in S)$ is referred to in the literature as a \emph{Poisson-Binomial distribution}. 
Let also $\Hyp(N,M,a)$ denote a \emph{hypergeometric random variable} counting the number of successes when a sample of size $a$ is drawn without replacement from a population of size $N$ which contains $M$ successes.
Given two random variables $X,Y$, we write $X\sim Y$ when $X$ and $Y$ have the same distribution.

In this article we shall be concerned with a model to generate random multigraphs. 
A \emph{multigraph} is a graph which is allowed to have multiple edges between pairs of vertices, and we shall be concerned with a random process which generates such graphs. 
Random processes which generate (simple) graphs are referred to as \emph{random graphs}, and constitute a classical topic of probabilistic combinatorics.  The study of random graphs is 
more than sixty years old and was initiated  by Gilbert~\cite{Gilbert} and Erd\H{o}s \& R\'enyi~\cite{ER1, ER2}, 
who introduced two quintessential random graph models. 
The first, and perhaps most studied, model is the \emph{binomial random graph},  denoted $\mG(n,p)$, which is obtained by independently introducing an edge between each pair of vertices in an empty graph on $n$ vertices with probability $p$.
The second model is  the \emph{uniform random graph},  denoted $\mG(n,m)$, which is obtained by choosing a graph uniformly at random from the set consisting of all graphs on the vertex set $[n]$ having $m$ edges. In other words, $\mG(n,m)$ is a subset of $\binom{\binom{[n]}{2}}{m}$ chosen uniformly at random. 
Since their introduction, the two models have been studied extensively and have been generalized in a plethora of ways, thus generating a  substantial amount of literature. 
We refer the reader to~\cite{Bollobas, Durrett, Frieze_Karonski, JLR, Palmer, Hofstad} for excellent textbooks on the theory of random graphs.    

We shall also be concerned with two  extension of the above-mentioned random graph models to hypergraphs.  
The \emph{binomial random $k$-uniform hypergraph}, denoted $\mathbb{H}_{n,q;k}$, is a $k$-uniform hypergraph $\mH$ whose hyperedges are obtained 
by independently including each element  $F\in\binom{[n]}{k}$ in $\mH$ with probability $q$. 
The \emph{uniform random $k$-uniform hypergraph}, denoted $\mathbb{H}_{n,m;k}$, is a $k$-uniform 
hypergraph $\mH$ whose hyperedges are the sets obtained by choosing an element from $\binom{\binom{[n]}{k}}{m}$ uniformly at random  (see~\cite[Chapter~14]{Frieze_Karonski}).

Random multigraph  models have also  attracted considerable attention due to the fact that they are more appropriate for ``real world" problems. 
For instance, the need for considering models which allow  multiple edges between pairs of vertices  arises in the study of social networks, where interactions  between individuals may become apparent in different contexts, or time periods (see, for example,~\cite{Shafie}). 
There is a significant amount of literature on random multigraphs, and   
several models have been proposed in various contexts. Examples  include  a \emph{configuration model} (see~\cite[Chapter~7]{Hofstad}), as well as several generalisations and variations of the binomial random graph and the uniform random graph (see, for example,~\cite{Anastos, Chen_Zhang_Li, Federico, Godehardt_Annals, Goldschmidt_Kreacic, Ranola}, among many others). 
In the present article we introduce and investigate the following random multigraph model.

\subsection{The model}\label{sec:model}

A \emph{binomial random multigraph}, denoted $\mG(n,p\,;\,\mH)$, is generated as follows: 
Fix a positive integer $n$, a real number $p\in (0,1]$, and a (perhaps random)
hypergraph $\mH$ on $[n]$. 
Begin with an empty graph on $n$ vertices, which are labelled by the set $[n]$. 
For every set $H\in \mH$, choose, independently from previous choices, a doubleton from $H$, say $D = \{i,j\} \in \binom{H}{2}$, uniformly at random and then introduce an edge between the vertices $i$ and $j$ in the graph with probability $p$, where each edge is introduced independently of all other edges.

\subsection{Relation to existing models}

When $\mH= \binom{[n]}{2}$, then $\mG(n,p\,;\,\binom{[n]}{2}) = \mG(n,p)$, and therefore the binomial random graph is a special case of our model. 
Observe that, when $\mH$ contains hyperedges such that $|H|\ge 3$, the multigraph 
$\mG(n,1\,;\, \mH)$ is still random, in contrast to $\mG(n,1)$ which is deterministic. 

A straightforward way to generalize $\mG(n,p\,;\,\mH)$ is to   sample $m$ times (with or without replacement) a doubleton from each element of $\mH$, and then introduce an edge between the corresponding pairs of vertices in the sample with probability $p$. 
Special cases of the model, when the sampling is \emph{with replacement}, have been reported in the literature. For example, 
the model $\mG\left(n,p\,;\, \binom{[n]}{2}^{(\times m)}\right)$ 
has been investigated in~\cite{Godehardt_Annals, Harris_Godehardt_Horsch}, in the setting of hypotheses testing for clustering and classification, and may be seen as a natural generalisation of the binomial random graph to multigraphs having at most $m$ edges between each pair of vertices.  When the sampling is \emph{without replacement} then, for every $H\in \mH$, we sample $m\le \binom{|H|}{2}$ distinct doubletons from $H$ uniformly at random, and then we flip a coin for every doubleton in the sample to decide whether we join the corresponding pairs of vertices in the graph. This model appears to be new. 
Observe that the $m$ distinct doubletons that are chosen uniformly at random from $\binom{H}{2}$ correspond to a random sub-graph of the complete graph on the vertex set $H$, and the corresponding model, with $p=1$, may be seen as a ``randomized version" of the binomial random motif graph, which is introduced   in~\cite{Anastos}.

One could also sample a different number of doubletons from each element of $\mH$, and even randomize this number. A bit more precisely, let $\mH$ be a hypergraph, and suppose that 
$\{X_H\}_{H\in\mH}$ are mutually independent random variables, indexed by the edges of $\mH$, taking values in the set of non-negative integers $\{0,1,2,\ldots\}$.
Let $\mbH{\{X_H ; H\in\mH\}}$ denote the (random) hypergraph that contains each  element $H\in\mH$  exactly $X_H$ times. That is, we draw from $X_H$, for each $H\in \mH$, and then create a hypergraph 
that contains each edge $H\in\mH$ exactly $X_H$ times. 
The model $\mG\left(n,1\,;\, \mbH{\{X_B; B\in \binom{[n]}{2}\}}\right)$  is introduced  in~\cite{Chen_Zhang_Li}, as a mutligraph analogue of the binomial random graph model.
The model $\mG\left(n,1\, ; \, \mbH{\{P_B\, ; \, B\in \binom{[n]}{2}\}}\right)$, where $\{P_B\}_{B\in \binom{[n]}{2})}$ are independent Poisson random variables of given mean is introduced in~\cite{Ranola}, as a model to study biological networks.  All the aforementioned models  can be considered in a more general setting, where $\binom{[n]}{2}$ is replaced by some other  hypergraph on $[n]$. Such models correspond to a model $\mG(n,p\, \, ; \, \,\mH)$ for which the hypergraph $\mH$ is random, but one could consider the model with more ``exotic" random hypergraphs.
For example,   one could consider the model 
$\mG(n,p\, ; \, \mathbb{H}_{n,k,q})$, which appears to be new.
The model $\mG(n,p\, ; \, \mbH_{n,m;k})$ reduces to $\mG(n,m)$, for $p=1,k=2$. 

In this article we shall mainly be concerned with the model $\mG(n,p\, ; \,\binom{[n]}{k})$, but we hope that we will be able to report on various  other models in the future.

\subsection{Main results}

Given  a multigraph property, say $\mP$, we write $G\in\mP$ when the multigraph $G$ posses property $\mP$. 
A ``typical" question in the theory of the binomial random graph model $\mG(n,p)$ begins with a particular graph property, say $\mP$, and asks for values $\{p_n\}_n$ for which $\Prb(\mG(n,p_n)\in \mP)$ tends to one, as n tends to infinity. 
Such questions involve the notion of \emph{thresholds} and 
one remarkable  discovery in the theory of random graphs, which is due to Erd\H{o}s and R\'enyi, is the fact that several graph properties admit a threshold. That is, for several graph properties $\mP$, there exists a sequence $\{p_n^{\ast}\}_n$ such that the probability that a random graph $G_n\in \mG(n,p_n)$ possesses $\mP$ tends to zero, when $p_n\ll p_n^{\ast}$, and tends to one when $p_n\gg p_n^{\ast}$. In other words, there is a ``phase transition" in this limiting probability which transits from the property being very unlikely to the property becoming almost certain. In this article we shall be concerned with 
thresholds in binomial random multigraphs. 
Before being more precise, let us first recall some definitions. 

A (multi)graph property $\mP$ is called \emph{non-trivial} if the empty graph does not possess $\mP$, while the complete graph does possess $\mP$. 
A (multi)graph property $\mathcal{P}$ is called \emph{monotone increasing} 
if $G\in \mathcal{P}$ implies that $G+e \in\mathcal{P}$, i.e., the property $\mathcal{P}$ 
is not destroyed after adding an edge $e$ to the graph $G$. Similarly, a (multi)graph property $\mathcal{P}$ is called \emph{monotone decreasing}  if $G\in \mathcal{P}$ implies that $G- e \in\mathcal{P}$, i.e., the property $\mathcal{P}$ 
is not destroyed by removing an edge $e$ from the graph $G$.
Clearly, a property $\mP$ is monotone increasing if and only if its complement is monotone decreasing, and thus any statement about an increasing property translates to a statement about a  decreasing one. 
Examples of monotone increasing properties are connectivity, Hamiltonicity, triangle-containment, among others. An example of a monotone decreasing multigraph property is the graph being simple. 
For deterministic hypergraphs, the corresponding notion of a threshold is analogous to the notion 
of a threshold in $\mG(n,p)$.

\begin{definition}[Threshold functions]
Let  $\{\mH_n\}_{n\ge 2}$ be a sequence of hypergraphs on $[n]$, and fix  a monotone increasing (multi)graph property $\mP$.  
A \emph{threshold function} for $\mP$, in the model $\mG(n,p\, ; \,\mH_n)$, is a sequence $\{c_n^{\ast}\}_n$ such that 

\[
\lim_{n\to\infty} \Prb(\mG(n,p_n\, ; \,\mH_n) \in\mP) = \begin{cases}
0,&\text{ if } \, \, p_n \ll c_n^{\ast}    \\
1,&\text{ if }\,\,  p_n \gg c_n^{\ast}  .
\end{cases}
\]
\end{definition}

However, when the hypergraph is random, the corresponding notion of a threshold 
has to take into account the underlying ``randomness of the hypergraph". 
For example,  thresholds in $\mG(n,p\, ; \,\mathbb{H}_{n,q;k})$ are defined as follows.

\begin{definition}[Threshold functions in $\mG(n,p\, ; \,\mathbb{H}_{n,q;k})$]
Let  $k\ge 3$ be fixed, and let $\mP$ be  a monotone increasing (multi)graph property.  
A \emph{threshold function} for $\mP$, in the model $\mG(n,p\, ; \,\mathbb{H}_{n,q;k})$, is a sequence $\{c_n^{\ast}\}_n$ such that 

\[
\lim_{n\to\infty} \Prb(\mG(n,p_n\, ; \,\mathbb{H}_{n,q_n;k}) \in\mP) = \begin{cases}
0,&\text{ if } \, \, p_n\cdot q_n \ll c_n^{\ast}    \\
1,&\text{ if }\,\,  p_n\cdot q_n \gg c_n^{\ast}  .
\end{cases}
\]
\end{definition}

A central line of research in the theory of  random graphs is concerned with  the 
investigation of thresholds for various graph properties. A well-known result,  
due to Bollob\'as and Thomason
(see~\cite[Theorem~1.7]{Frieze_Karonski}), 
asserts that every monotone increasing graph property admits a threshold in $\mG(n,p)$. We believe that the Bollob\'as--Thomason theorem holds true for several ``well-behaved" binomial random multigraphs.   For example, we believe that the following holds true.

\begin{conjecture}\label{Bollobas_thomason}
Let $k\ge 3$ be fixed, and let $\mP$ be a non-trivial monotone increasing (multi)graph property. Then $\mP$ has a threshold in  $\mG(n,p\, ; \,\binom{[n]}{k})$. 
\end{conjecture}

%The proof of Theorem~\ref{Bollobas_thomason} is obtained along the same lines as the proof of the result in $\mG(n,p)$ (see~\cite[Theorem~1.7]{Frieze_Karonski}). Let us remark that the proof actually works for more general hypergraphs $\{\mH_n\}_n$, provided they satisfy $\deg_{\mH_n}(B) = O(n^{k-2})$, for all $B\in\binom{[n]}{2}$ (see Lemma~\ref{lim_one} below). 

We are unable to prove Conjecture~\ref{Bollobas_thomason}, though we  will provide some evidence for its validity. Our first result implies that several increasing properties admit a ``gapped threshold" in $\mG(n,p\, ; \,\binom{[n]}{k})$.

\begin{theorem}\label{gap_threshold}
Let $k\ge 3$ be fixed, and let $\mP$ be a non-trivial monotone increasing (multi)graph property which admits a threshold in $\mG(n,p)$. 
Then there exist sequences $\{b_n^{\ast}\}_n$ and $\{c_n^{\ast}\}_n$ such that 
\[
\lim_{n\to\infty} \Prb\left(\mG(n,p_n\, ; \,\binom{[n]}{k}) \in\mP\right) = \begin{cases}
0,&\text{ if } \, \, p_n \ll b_n^{\ast}    \\
1,&\text{ if }\,\,  p_n \gg c_n^{\ast}  .
\end{cases}
\]
\end{theorem}

Observe that Theorem~\ref{gap_threshold} does not provide a ``gapped threshold" for all increasing properties $\mP$, but only for properties that admit thresholds in $\mG(n,p)$. For example, it provides no information for the property  $\mP=\,$``the graph is simple". 

The reader familiar with $\mG(n,p)$ will have already noticed several analogies between the binomial random graph and the binomial random multigraph models, and we will illustrate this analogy further. 
Our next two results show that the probability that the latter model possesses 
an increasing property is, similarly to the former model,  an increasing function of $p$, as well as an increasing function of $\mH$.

\begin{theorem}\label{monot:1}
Fix a positive integer $n$, two real numbers $p_1,p_2\in [0,1]$, and a hypergraph $\mH$ on $[n]$. 
If $p_1\le p_2$ then, for every monotone increasing property $\mathcal{P}$, it holds
\[
\Prb(\mG(n,p_1\, ; \, \mH) \in \mathcal{P}) \le \Prb(\mG(n,p_2\, ; \, \mH) \in \mathcal{P}) \, .
\]
\end{theorem}

\begin{theorem}\label{monot:2}
Fix a positive integer $n$, a real number $p\in [0,1]$, and two hypergraphs $\mH_1,\mH_2$ on $[n]$. 
If $\mH_1 \subset\mH_2$ then, for every monotone increasing property $\mathcal{P}$, it holds
\[
\Prb(\mG(n,p\, ; \, \mH_1) \in \mathcal{P}) \le \Prb(\mG(n,p\, ; \, \mH_2) \in \mathcal{P}) \, .
\]
\end{theorem}

We also provide some additional evidence for the validity of Conjecture~\ref{Bollobas_thomason}, by 
determining the limiting probability, $\Prb(\mG(n,p,\mH)\in \mP)$, for some graph properties $\mP$ in various random multigraph models. 
We begin with a  result  which is related to the ``simplicity threshold". 

\begin{theorem}\label{thm:simple0} 
Let $\mP$ denote the property ``the graph is simple", and let $k\ge 3$ be fixed. The following hold true.
\begin{enumerate}
\item If $p_n\ll \frac{1}{n^k}$, then  
\[
\lim_{n\to\infty} \Prb\left(\mG(n,p_n\, ; \,\binom{[n]}{k})\in\mP\right) = 1 \, .
\]
\item If $p_n\gg \frac{1}{n^{k-1}}$, then  
\[
\lim_{n\to\infty} \Prb\left(\mG(n,p_n\, ; \,\binom{[n]}{k})\in\mP\right) = 0 \, .
\]
\end{enumerate}
\end{theorem}

In other words, Theorem~\ref{thm:simple0} provides estimates on the threshold function for the property ``the graph is simple".
We are unable to determine an exact threshold for this property. It is easy to see that a threshold 
for the property ``the graph contains at least one edge" is $\frac{1}{n^k}$ (see Lemma~\ref{all_edges} below) and, since it is not unreasonable to expect that $\mG(n,p,\binom{[n]}{k})$ contains multiple edges as long as it has some edges, perhaps a simplicity threshold equals $\frac{1}{n^k}$ as well.

Theorem~\ref{thm:simple0} has no analogue in $\mG(n,p)$ but, for graph properties that make sense in both models, there seems to be an analogy between thresholds in $\mG(n,p)$ and $\mG(n,p\, ; \,\binom{[n]}{k})$. However, in the latter case the  calculations involving expectations and variances of various random 
variables are significantly more complex. 
A reason for this ``complexity" is that the model $\mG(n,p\, ; \,\binom{[n]}{k})$,  for $k\ge 3$, exhibits certain ``dependencies"  which are not present in  $\mG(n,p)$. 
Let us illustrate this a bit further. 

For $B\in\binom{[n]}{2}$, let $E_B$ be the number of edges between the two vertices in $B$. Then, in the model  $\mG(n,p\, ; \,\binom{[n]}{k}), k\ge 3$,  the random variables $\{E_B\}_{B\in\binom{[n]}{2}}$ are correlated; a fact that complicates the calculation of moments and probabilities. 
For example, in order to compute the probability that a particular triangle, say $T$, is present in $\mG(n,p_n\, ; \,\binom{[n]}{k})$,  we have to compute the probability that the two vertices of every  $B\in \binom{T}{2}$ are joined by at least one edge. The latter probability depends, among other things, on the probability 
that the coin-flips corresponding to 
the elements in $\mH_T = \{H\in\binom{[n]}{k} : T\subset H\}$ resulted in $\ell$ edges between the vertices of $T$, where $\ell \in\{0,1,\ldots,|\mH_T|\}$, and we are unable to find a neat way to calculate such probabilities for  $k> 3$. When $k=3$, then $\mH_T$ consists of a single element, a fact that simplifies the calculations, as we show in the proof of Theorem~\ref{thm:triangle_rand}  below.  
The proofs of the following results illustrate further the aforementioned dependencies.

\begin{theorem}\label{thm:connected}
The property ``the graph is connected" has threshold function $\frac{\log(n)}{n^2 }$ in $\mG(n,p\, ; \,\binom{[n]}{3})$.
\end{theorem}

Given Theorem~\ref{thm:connected} and bearing in mind the corresponding result in $\mG(n,p)$, it may be reasonable to expect that a connectivity threshold in $\mG(n,p\, ; \,\binom{[n]}{k}), k\ge 4,$ is $\frac{\log(n)}{n^{k-1}}$.

Having determined the connectivity threshold in $\mG(n,p\, ; \,\binom{[n]}{3})$, a next step may be to  
determine the ``triangle containment" threshold in this model. 
A \emph{triangle} in a multigraph $G$, with vertex set $[n]$, is a triplet of vertices $T\in\binom{[n]}{3}$ such that all pairs of vertices from $T$ are joined with at least one edge. 
A threshold for ``triangle containment" in $\mG(n,p\, ; \,\binom{[n]}{3})$ should be $\frac{1}{n^2}$, but we are unable to prove it. We can determine the order of the expected number of triangles in $\mG(n,p\, ; \,\binom{[n]}{3})$, but are unable to determine the order of its variance, due to the aforementioned ``dependencies". 
The next result, combined with standard first moment arguments, implies that the expected number of triangles in $\mG(n,p\, ; \, \mbH_{n,q;3})$ goes to zero, when $p_n\cdot q_n \ll \frac{1}{n^2}$.

\begin{theorem}\label{thm:triangle_rand}
The expected number of triangles in $\mG(n,p\, ; \, \mbH_{n,q;3})$ is equal to 
\[
\binom{n}{3}\cdot \left\{ (1-p\cdot q)\cdot \left( \Prb\left(\Bin\left(n-3,\frac{pq}{3}\right)\ge 1\right) \right)^3 + p\cdot q \cdot \left( \Prb\left(\Bin\left(n-3,\frac{pq}{3}\right)\ge 1\right) \right)^2 \right\}\, .
\]
\end{theorem}

Note that  the expected number of triangles in $\mG(n,p\, ; \,\binom{[n]}{3})$ are obtained from  Theorem~\ref{thm:triangle_rand} for $q=1$. The expected number of triangles in $\mG(n,p\, ; \, \mbH_{n,m;3})$ 
is of  similar form and illustrates the, rather intuitive, idea that $\mG(n,p\, ; \, \mbH_{n,q;3})$ and 
$\mG(n,p\, ; \, \mbH_{n,m;3})$ should be ``close" to each other when $q=\frac{m}{\binom{n}{3}}$. 

\begin{theorem}\label{thm:triangle_rand2}
The expected number of triangles in $\mG(n,p\, ; \, \mbH_{n,m;3} )$ is equal to 
\[
\binom{n}{3}\cdot\left\{\left(1- p\cdot\frac{m}{\binom{n}{3}}\right) \cdot    \left(\Prb\left(\Bin(W,p)\ge 1\right) \right)^3   + p \cdot\frac{m}{\binom{n}{3}}\cdot    \left(\Prb\left(\Bin(W,p)\ge 1 \right)\right)^2  \right\} \, , 
\]
where $W\sim\Hyp(\binom{n}{3}, n-3, m)$. 
\end{theorem}

The purpose of our final result is to illustrate a bit further the above-mentioned ``complexities of calculations" for values of $k$ that are larger than three. 

\begin{theorem}\label{thm:triangle_four}
The expected number of triangles in $\mG(n,p\, ; \,\binom{[n]}{4})$ is equal to 
\[
\binom{n}{3}\cdot \left\{ \left(1- \sum_{i=0}^{2} p_{0i}^{(n-3)}\right) + 
\sum_{i=0}^{2} p_{0i}^{(n-3)} \cdot \left(\Prb\left( \Bin\left(\binom{n-3}{2}, \frac{p}{6} \right)\ge 1 \right)\right)^{3-i}  \right\} \, ,
\]
where $p_{0i}^{(n-3)}, i=0,1,2,$ is the $(0,i)$-th element of the $(n-3)$-th power of the following transition matrix 
of a homogeneous  Markov chain with state space $\{0,1,2,3\}$:
\[
P = 
\begin{blockarray}{ccccc}
 & 0 & 1 & 2 & 3 \\
\begin{block}{c[cccc]}
  0 & 1-p/2 & p/2 & 0 & 0  \\
  1 & 0    & 1-p/3    & p/3 & 0  \\
  2 & 0 & 0 & 1-p/6 & p/6  \\
  3 & 0 & 0 & 0 & 1  \\ 
\end{block}  
\end{blockarray}  \, \, .
 \]
 
\end{theorem}

\subsection{Organisation} 

The remaining part of our article is organised as follows. 
In Section~\ref{sec:mixture} we show that a 
$\mG(n,p\, ; \,\mH)$ model is a mixture of models of the form $\mG(n,p\, ; \,\mD)$, where each  $\mD$ is a particular $2$-uniform hypergraph on $[n]$.  In Section~\ref{sec:degree} 
we collect some observations related to the degree distribution of vertices in various binomial random multigraph models. In Section~\ref{sec:monotone} we prove Theorems~\ref{monot:1} and~\ref{monot:2}.  
The proof of Theorem~\ref{gap_threshold} is given in Section~\ref{sec:proof_B_T}, 
and Theorems~\ref{thm:simple0} and~\ref{thm:connected} are proven in Section~\ref{sec:thresholds}. 
Finally, the proofs of Theorems~\ref{thm:triangle_rand}, ~\ref{thm:triangle_rand2} and~\ref{thm:triangle_four} are given in Section~\ref{sec:triangles}.

\section{The model as a mixture of random multigraphs}\label{sec:mixture}

%For every $B\in \binom{[n]}{2}$, let $m(B;\mD)$ be the multiplicity of $B$ in $\mD$. 
In this section we show that the model 
$\mG(n,p\, ; \,\mH)$ is a mixture of models of the form $\mG(n,p\, ; \,\mD)$, for a particular class of $2$-uniform hypergraphs $\mD$. 
This requires some extra piece of notation.

\begin{definition}[Shadow]
\label{shadow}
Let $\mH=\{H_1,\ldots,H_m\}$ be a hypergraph on $[n]$. 
Let $\mathscr{S}(\mH)$ denote the class  consisting of all $2$-uniform hypergraphs $\mD=\{D_1,\ldots,D_m \}$, having $m$ edges, which satisfy  
$D_i \in \binom{H_i}{2}$, for all $i=1,\ldots,m$. Clearly, if $\mH$ is $2$-uniform, then $\mathscr{S}(\mH) = \{\mH\}$. 
We refer to the class  
$\mathscr{S}(\mH)$ as the \emph{shadow} of the hypergraph $\mH$. 
\end{definition}

Observe that each $\mD\in\mathscr{S}(\mH)$ is a $2$-uniform hypergraph consisting of $|\mH|$ doubletons. 

\begin{theorem}\label{mixture}
Fix a positive integer $n$, a real number $p\in [0,1]$ and a hypergraph $\mH$ on $[n]$. Let $\mP$ be a graph property. 
Then 
\[
\Prb(\mG(n,p\, ; \,\mH) \in \mP) =\prod_{H\in \mH} \frac{1}{\binom{|H|}{2}}\cdot \sum_{\mD\in \mathscr{S}(\mH)} \Prb\left(\mG(n,p\, ; \, \mD) \in \mP \right) \, ,
\]
where $\mathscr{S}(\mH)$ is the shadow of $\mH$. 
\end{theorem}
\begin{proof}
The multigraph $\mG(n,p\, ; \,\mH)$ can be equivalently generated as follows: 
We first choose, independently from previous choices, a doubleton $D\in\binom{H}{2}$ uniformly at random, for each $H\in \mH$. 
This results in a collection consisting of $|\mH|$ doubletons, say $D_1,\ldots,D_{|\mH|}$, which form a $2$-uniform hypergraph, say  $\mathcal{D}\in\mathscr{S}(\mH)$. 
Then, given $\mD\in\mathscr{S}(\mH)$,  we flip a coin for every $D\in\mD$ to decide whether we put an edge between the vertices of  $D$, with each coin-flip being independent from all others. 
Hence, given $\mathcal{D}\in\mathscr{S}(\mH)$, 
the probability that $\mG(n,p\, ; \,\mH)\in\mP$ is equal to the probability that  $\mG(n,p\, ; \,\mD)\in\mP$. 
Notice that the probability of choosing a particular $\mathcal{D}\in\mathscr{S}(\mH)$ is equal to $\prod_{H\in \mH}\, \frac{1}{\binom{|H|}{2}}$. 
The result follows from the law of total probability, 
upon conditioning on $\mD\in\mathscr{S}(\mH)$. 
\end{proof}

\section{Degree distribution}\label{sec:degree}

In this section we collect some observations 
regarding the degree distribution of a vertex, and the number of edges between a pair of vertices in  binomial random multigraphs. 
%The \emph{neighborhood} of $i$, denoted $\mN_i$, is the set consisting of those vertices that are adjacent to $i$. The cardinality of $\mN_i$ is denoted $\nu_i$, for $i\in [n]$. 
%Observe that, since $\mG(n,p;\mH)$ is a multigraph, 
%it holds $\deg(i)\ge \nu_i$, with the inequality being strict when multiple edges are present. 

\begin{proposition}\label{prop:1}
Let $G\in \mG(n,p\, ; \,\mH)$ and fix a vertex $v\in [n]$. 
Let $\mH_v=\{H\in \mH : v\in H\}$ and, for each $H\in\mH_v$, set $p_H = \frac{2p}{|H|}$. 
Then $\deg_G(v) \sim \mB(p_H ; H\in \mH_v)$. 
\end{proposition}
\begin{proof}
Observe that vertex $v$ has a potential edge for every $H\in \mH_v$. 
Now, given $H\in\mH_v$, the probability of choosing a doubleton $D\in\binom{H}{2}$ such that $i\in D$ is equal to 
$\frac{|H|-1}{\binom{|H|}{2}}$. Hence the probability that vertex $v$ is adjacent to some vertex  
$w\in H$, is equal to $\frac{|H|-1}{\binom{|H|}{2}}\cdot p$. 
In other words, the number of edges adjacent to 
vertex $v$ equals the number of successes 
in $|\mH_v|$ independent trials, where  trial $H\in \mH_v$ has probability 
$\frac{|H|-1}{\binom{|H|}{2}}\cdot p = \frac{2p}{|H|}$ of being a success. The result follows. 
\end{proof}

\begin{corollary}
\label{cor:1}
Fix a positive integers $n\ge k\ge 2$, a real number $p\in [0,1]$, and let 
$\mH$ be a $k$-uniform hypergraph on $[n]$. Let $G\in \mG(n,p\, ; \,\mH)$ and fix a vertex  $v\in [n]$. 
Let $\mH_v=\{H\in \mH : v\in H\}$. 
Then 
$\deg_G(v) \sim \Bin(|\mH_v|, \frac{2p}{k})$. 
\end{corollary}
\begin{proof}
Immediate from Proposition~\ref{prop:1}. 
\end{proof}

In other words, the number of edges that are incident to each vertex is a binomial random variable whose number of successes depend on the number of $k$-sets that contain the vertex. 
Similar results holds true for the number of edges between a fixed pair of vertices.

\begin{proposition}\label{prop:2} 
Let $G\in \mG(n,p\, ; \,\mH)$ and fix two distinct vertices $u,v\in [n]$. Let $\mH_{u,v}=\{H\in \mH : \{u,v\}\subset H\}$. 
Then the number of edges between vertices $u$ and $v$ is a $\mB(p_H ; H\in \mH_{u,v})$ random variable, where 
$p_H = \frac{p}{\binom{|H|}{2}}$. 
\end{proposition}
\begin{proof}
For every $H\in \mH_{u,v}$, the probability of choosing the doubleton $\{u,v\}$ is equal to 
$\frac{1}{\binom{|H|}{2}}$. Therefore, the probability that vertex $u$ 
is adjacent to vertex $v$ in trial $H\in \mH_{u,v}$ is equal 
to $\frac{1}{\binom{|H|}{2}}\cdot p$. The result follows. 
\end{proof}

\begin{corollary}
\label{cor:2}
Fix a positive integers $n\ge k\ge 2$, a real number $p\in [0,1]$, and let 
$\mH$ be a $k$-uniform hypergraph on $[n]$. Let $G\in \mG(n,p\, ; \,\mH)$ and fix two distinct vertices $i,j\in [n]$. 
Let $\mH_{i,j}=\{H\in \mH : \{i,j\}\subset H\}$. 
Then the number of edges between the vertices $i$ and $j$ is a
$\Bin\left(|\mH_{i,j}|\,,\, \frac{p}{\binom{k}{2} }\right)$ random variable. 
\end{corollary}
\begin{proof}
Immediate from Proposition~\ref{prop:2}. 
\end{proof}

\begin{proposition}
Fix positive integers $n\ge k\ge 2$ and real numbers $p,q\in [0,1]$. 
Let $v$ be a vertex in $G\in \mG(n,p\,;\;\mbH_{n,q;k})$. 
Then $\deg_G(v)\sim \Bin\left(\binom{n-1}{k-1}, \frac{2p}{k} q\right)$.
\end{proposition}
\begin{proof}
Let $X$ be the number of edges in $\mbH_{n,q;k}$ that contain vertex $v$. 
Then $X\sim \Bin\left(\binom{n-1}{k-1}, q\right)$. 
Now, given $X$, it follows from Corollary~\ref{cor:1} that $\deg_G(v)\sim \Bin(X, \frac{2p}{k})$, and the result follows. 
\end{proof}

In the case of $\mG(n,p\,;\,\mbH_{n,m;k})$, the degree of a vertex is a mixture of binomial random variables. 
Recall that $\Hyp(N,M,a)$ denotes a hypergeometric random variable. 

\begin{proposition}
Fix positive integers $n\ge k\ge 2$, a real number $p\in [0,1]$ and a positive integer $m\le \binom{n}{k}$. 
Let $v$ be a vertex in $G\in \mG(n,p,\mbH_{n,m;k})$. 
Then $\deg_G(v)\sim \Bin\left(W, p \right)$, where $W\sim\Hyp\left(\binom{n}{k},\binom{n-1}{k-1},m \right)$. 
\end{proposition}
\begin{proof}
Let $W$ be the number of edges in $\mbH_{n,m;k}$ that contain vertex $v$. 
Note that is holds  $W\sim\Hyp\left(\binom{n}{k},\binom{n-1}{k-1},m \right)$. 
Now, given $W$, it follows from Corollary~\ref{cor:1} that $\deg_G(v)\sim \Bin(W, \frac{2p}{k})$.
In other words, we have 
\begin{eqnarray*}
\Prb(\deg_G(v)= \ell) &=& \sum_{j\ge \ell} \Prb(\deg_G(v)=\ell\,|\, W=j)\cdot \Prb(W=j) \\
                                   &=& \sum_{j\ge \ell} \Prb\left(\Bin(\ell, \frac{2p}{k}) = j\right) \cdot \Prb(W=j) 
\end{eqnarray*}
and the result follows. 
\end{proof}

\section{Monotonicity}\label{sec:monotone}

In this section we prove Theorems~\ref{monot:1} and~\ref{monot:2}. 

\begin{proof}[Proof of Theorem~\ref{monot:1}]
We may assume that $p_1<p_2$; say, $p_2 = p_1 + (1-p_1)\cdot p$, for some $p\in (0,1)$. 
Begin with generating a graph $G\in\mG(n,p_1\, ; \,\mH)$. 
That is, choose, independently from previous choices, a doubleton $D\in \binom{H}{2}$ uniformly at random from every $H\in\mH$. 
This results in $|\mH|$ doubletons, say $D_1,\ldots,D_{|\mH|}$, and for every such doubleton 
we, independently, flip a coin to decide whether we introduce an edge between the 
corresponding vertices in the graph; where the outcome of each coin is a success with probability $p_1$. 
A doubleton $D\in \{D_1,\ldots,D_{|\mH|}\}$ is said to be \emph{bad} if 
the outcome of the corresponding coin-flip was a failure; this happens with probability $1-p_1$. 
Now, for every bad doubleton from $\{D_1,\ldots,D_{|\mH|}\}$, 
we independently introduce an edge between the corresponding vertices 
in the graph $G$ with probability $p$. Let $G^{\prime}$ be the resulting graph. 
Clearly, it holds $G\subset G^{\prime}$ and 
$G^{\prime} \in \mG(n,p_2\, ; \,\mH)$. If $G\in\mathcal{P}$, 
then it also holds $G^{\prime}\in\mathcal{P}$, and the result follows. 
\end{proof}

\begin{proof}[Proof of Theorem~\ref{monot:2}]
The proof is similar to the proof of Theorem~\ref{monot:1}; so we briefly sketch it. 
Begin with generating a graph $G\in\mG(n,p\, ; \,\mH_1)$. Let $\mH = \mH_2\setminus \mH_1$. 
For every $H\in\mH$ choose a doubleton, say $D=\{i,j\}\in\binom{H}{2}$ uniformly at random, 
and then introduce an edge between the vertices $i$ and $j$ in $G$ with probability $p$. Let 
$G^{\prime}$ denote the resulting graph. Clearly, it holds $G\subset G^{\prime}$ and 
$G^{\prime}\in \mG(n,p\, ; \,\mH_2)$. 
\end{proof}

\section{Proof of Theorem~\ref{gap_threshold}}
\label{sec:proof_B_T}

In this section we prove Theorem~\ref{gap_threshold}. 
Let us begin with determining the threshold of the simplest possible increasing graph property. 

\begin{lemma}\label{all_edges}
The property ``the graph contains at least one edge" has threshold $\frac{1}{|\mH_n|}$ in $\mG(n,p\, ; \,\mH_n)$. 
\end{lemma}
\begin{proof}
Let $G_n\in \mG(n,p_n\, ; \, \mH_n)$, and observe that $G_n$ is empty if and only if all $|\mH_n|$ coin-tosses corresponding to doubletons chosen from elements of $\mH$  
were failures. In other words, it holds  
\[
\Prb(G_n \text{ is empty}) = \Prb(\Bin(|\mH_n|, p_n)=0) = (1-p_n)^{|\mH_n|} \, .
\] 
Suppose first that $p_n\ll \frac{1}{|\mH_n|}$. Then, using the inequality $1-x\ge e^{-2x}$, for $x\in[0,1/2]$, we obtain 
\[
\Prb(G_n \text{ is empty})=(1-p_n)^{|\mH_n|}  \ge e^{-2p_n|\mH_n|} \to 1 \, , \text{ as } n\to\infty \, .
\]
Assume now that $p_n\gg \frac{1}{|\mH_n|}$. Then, using the inequality $1-x\le e^{-x}$, for $x\in \mathbb{R}$, we obtain
\[
\Prb(G_n \text{ is empty})=(1-p_n)^{|\mH_n|}  \le e^{-p_n|\mH_n|} \to 0 \, , \text{ as } n\to\infty \, .
\]
Hence $\frac{1}{|\mH_n|}$ is the desired threshold. 
\end{proof}

We will also need the following lemma.

\begin{lemma}\label{lim_one}
Let $k\ge 3$ be fixed, and 
let $\{\mH_n\}_n$ be a sequence of $k$-uniform  hypergraphs on $[n]$  
such that $\deg_{\mH_n}(B)=\Theta(n^{k-2})$, for all $n$ and all $B\in\binom{[n]}{2}$. 
Let $\mP$ denote the monotone increasing property ``all pairs of vertices are joined by at least one edge". 
Then 
\[
\lim_{n\to +\infty} \, \Prb(\mG(n,1\, ; \,\mH_n) \in \mP) = 1 \, .
\]
\end{lemma}
\begin{proof}
For every $B=\{i,j\}\in\binom{[n]}{2}$, let $Y_B$ denote the indicator of the event ``vertices $i$ and $j$ are joined by at least one edge", and set $Y= \sum_{B\in\binom{[n]}{2}}Y_B$. 
Then $Y$ counts the number of pairs of vertices from $[n]$ that are joined by at least one edge. Note note that  $Y< \binom{n}{2}$ implies that there exists $B\in\binom{[n]}{2}$ such that $Y_B=0$. 
Hence the union bound combined with Corollary~\ref{cor:2} imply that there exists constant $C$, which does not depend on $n$, such that  
\begin{eqnarray*}
\Prb\left(Y < \binom{n}{2}\right) &\le& \sum_{B\in\binom{[n]}{2}} \Prb(Y_B =0) \,
\le\, \binom{n}{2}\cdot \left( 1 - \frac{1}{\binom{k}{2}} \right)^{C\cdot n^{k-2} }\\
&\le& \binom{n}{2}\cdot e^{-\frac{C\cdot n^{k-2}}{\binom{k}{2}}} \, =\,  
e^{\log(\binom{n}{2})  -\frac{C\cdot n^{k-2}}{\binom{k}{2}}} \to 0 \, ,
\end{eqnarray*}
as desired. 
\end{proof}

\begin{corollary}\label{randd}
Let $\mD\in\mathscr{S}(\binom{[n]}{k})$ be chosen uniformly at random. 
Then $\Prb(\mD \, \text{is complete}) \to 1$, as $n\to\infty$. 
\end{corollary}
\begin{proof}
Set $\mH_n=\binom{[n]}{k}$ and note that $\deg_{\mH_n}(B) = \binom{n-2}{k-2}$ for all $B\in\binom{[n]}{2}$.  Let $\mP$ denote the monotone increasing property ``all pairs of vertices are joined by at least one edge". 
Lemma~\ref{lim_one} and Theorem~\ref{mixture} imply  that 
   \[
    1 = \lim_{n\to\infty}  \, \Prb(\mG(n,1\, ; \,\mH_n) \in \mP)  = \lim_{n\to\infty}  \frac{1}{\binom{k}{2}^{|\mH_n|}}\sum_{\mD\in \mathscr{S}(\binom{[n]}{k})} 
   \Prb(\mG(n,1\, ; \, \mD) \in \mP)  \, .
   \]
   Now observe that $\Prb(\mG(n,1\, ; \, \mD) \in \mP)$ is equal to either $1$ or $0$, depending on whether $\mD$ is complete or not.  
   Hence, if $\mathbf{1}_{\mD}$ is the indicator of the event "$\mD$ is complete", it holds 
   \[
   \lim_{n\to\infty}  \frac{1}{\binom{k}{2}^{|\mH_n|}}\sum_{\mD\in \mathscr{S}(\binom{[n]}{k})} 
   \mathbf{1}_{\mD}  \, =\, 1 \, ,
   \]
   as desired. 
\end{proof}

We may now proceed with the proof of Theorem~\ref{gap_threshold}.

\begin{proof}[Proof of Theorem~\ref{gap_threshold}]
To simplify notation, let us set $\mH_{n,k}=\binom{[n]}{k}$. Let $\mP$ be an increasing multigraph property 
which admits a threshold in $\mG(n,p)$. Let $\{c_n^{\ast}\}_n$ be a threshold for $\mP$ in $\mG(n,p\, ; \,\binom{[n]}{2})$, and let 
$G_n\in\mG(n,p_n\, ; \,\mH_n)$.

Set $b_n^{\ast} = \frac{1}{n^k}$, and suppose that $p_n \ll b_n^{\ast}$. 
Since $\Prb(G_n \text{ is empty})\le \Prb(G_n\notin \mP)$, 
Lemma~\ref{all_edges} implies $\Prb(G_n\notin \mP) \to 1$, as $n\to\infty$, and the first statement follows. 

To prove the second statement, suppose that $p_n \gg  c_n^{\ast}$. Theorem~\ref{mixture} then yields
\begin{eqnarray*}
 \Prb(G_n\in \mP) &=& \binom{k}{2}^{-|\mH_n|} \sum_{\mD\in \mathscr{S}(\mH_n ) } \Prb(\mG(n,p_n\, ; \,\mD)\in \mP)\\
 &\ge& \binom{k}{2}^{-|\mH_n|} \sum_{\mD\in \mathscr{S}(\mH_n ): \mD \text{ complete} } \Prb(\mG(n,p_n\, ; \,\mD)\in \mP) \, .
\end{eqnarray*}
For every complete $\mD\in \mathscr{S}(\mH_n)$ it holds $\Prb(\mG(n,p_n\, ; \,\mD)\in \mP) \ge \Prb(\mG(n,p_n\, ; \, \binom{[n]}{2}) \in\mP)$, by Theorem~\ref{monot:2}, and hence 
\[
 \Prb(G_n\in \mP) \ge \binom{k}{2}^{-|\mH_n|} \sum_{\mD\in \mathscr{S}(\mH_n ): \mD \text{ complete} } \Prb(\mG(n,p_n\, ; \,\binom{[n]}{2}) \in\mP)  = \pi_n\cdot \Prb(\mG(n,p_n\, ; \,\binom{[n]}{2}) \in\mP)  \, 
\]
where $\pi_n$ is the probability that a randomly selected $\mD\in \mathscr{S}(\mH_n )$ is complete.  
The second statement follows from Corollary~\ref{randd} and the fact that $\Prb(\mG(n,p_n\, ; \,\binom{[n]}{2})\in\mP)\to 1$, as $n\to\infty$.  
\end{proof}

\section{Thresholds}\label{sec:thresholds}

\subsection{Proof of Theorem~\ref{thm:simple0}}

Let us begin with the following observation, whose proof is included for the sake of completeness.  

\begin{lemma}\label{bin_ineq}
Let $n\ge k\ge 3$ be positive integers, and let $\varepsilon\in (0,1]$. Suppose that $X\sim \Bin(c\cdot n^{k-2},\frac{1}{n^{k-1-\varepsilon}})$, for some constant $c$ which does not depend on $n$.  
Then, for large enough $n$, it holds that 
\[
\Prb(X\ge 2) \,\ge\, \frac{1}{n^{2-\varepsilon}} \, .
\]
\end{lemma}
\begin{proof}
Assume first that $\varepsilon<1$. 
Since $\varepsilon\in(0,1)$, it follows that $c\cdot n^{k-2}\le n^{k-1-\varepsilon}$, for large enough $n$, and therefore for such $n$ it holds  
\begin{eqnarray*}
\Prb(X\ge 2) \ge \Prb(X=2)  &=& \binom{c\cdot n^{k-2}}{2} \cdot \left(\frac{1}{n^{k-1-\varepsilon}} \right)^2\cdot \left(1-\frac{1}{n^{k-1-\varepsilon}}\right)^{c\cdot n^{k-2}-2} \\
&\ge& \left(\frac{c\cdot n^{k-2}}{2}\right)^2\cdot \left(\frac{1}{n^{k-1-\varepsilon}} \right)^2\cdot \left(1-\frac{1}{n^{k-1-\varepsilon}}\right)^{n^{k-1-\varepsilon}} \\
&\ge& \left(\frac{c\cdot n^{k-2}}{2 n^{k-1-\varepsilon}}\right)^2 \cdot \frac{1}{4} \\
&=& \left(\frac{c}{4}\right)^2 \cdot \frac{1}{n^{2-2\varepsilon}} 
\end{eqnarray*}
where in the second estimate we have used the inequality $\left(1-\frac{1}{x}\right)^x \ge \frac{1}{4}$, for $x\ge 2$.
Now, for $n$ large enough, it holds  $\left(\frac{c}{4}\right)^2 \cdot \frac{1}{n^{2-2\varepsilon}} \ge \frac{1}{n^{2-\varepsilon}}$, and the result follows. 

Suppose now that $\varepsilon=1$. 
If $c\ge 1$, then $\Prb(X\ge 2) \ge \Prb(\Bin(n^{k-2}, \frac{1}{n^{k-2}}) \ge 2)\ge 1/2$, since the median of a  $\Bin(n^{k-2}, \frac{1}{n^{k-2}})$ random variable is equal to its mean (see~\cite{Kaas}).  
If $c<1$ then, in the same way as above, we have 
\begin{eqnarray*}
\Prb(X\ge 2) &\ge& \left(\frac{c\cdot n^{k-2}}{2}\right)^2\cdot \left(\frac{1}{n^{k-2}} \right)^2\cdot \left(1-\frac{1}{n^{k-2}}\right)^{n^{k-2}} \\
&\ge& \left(\frac{c}{4}\right)^2  \, ,
\end{eqnarray*}
and the latter is larger than $\frac{1}{n}$, for large enough $n$.   
\end{proof}

We now proceed with the proof of Theorem~\ref{thm:simple0}. 
Recall, from Section~\ref{sec:mixture}, the definition of the shadow of a hypergraph $\mH$ on $[n]$, which is denoted $\mathscr{S}(\mH)$, and recall that, given $A\subset [n]$ and a (multi)hypergraph $\mH$ on $[n]$, we denote by $m(A;\mH)$  the multiplicity of $A$ in $\mH$.   

Recall also that in order to generate a multigraph $G_n\in\mG(n,p\, ; \,\mH)$ we randomly select 
$\mD\in\mathscr{S}(\mH)$ and then we flip a coin for each doubleton $D\in \mD$  to decide whether we introduce an edge between the two vertices of $D$. We refer to this coin-flipping  procedure as the \emph{coin-flips corresponding to $\mD$}.

Given $\mD\in\mathscr{S}(\binom{[n]}{k})$, for $k\ge 3$,  let 
\[
\mF_{\mD} = \left\{B\in \binom{[n]}{2} : m(B;\mD)\ge 2\right\}
\]
and observe that, since 
$|\mD|= \binom{n}{k} = O(n^k)$ and $k\ge 3$, it holds 
$\mF_{\mD}\neq\emptyset$, for $n$ large enough. 
Furthermore, for $n$ large enough, it holds 
$|\mF_{\mD}|= \Omega(n^2)$ for all $\mD\in\mathscr{S}(\binom{[n]}{k})$. Indeed, 
if this is not the case then there exists $\mD\in\mathscr{S}(\binom{[n]}{k})$ such that 
$|\mF_{\mD}|= O(n^{2-\varepsilon})$, for some $\varepsilon>0$, and  therefore the fact that 
$m(B;\mD)\le \binom{n-2}{k-2}$, for all $B\in\mF_{\mD}$,  yields 
\[
|\mD| \le O(n^{2-\varepsilon}) \cdot \binom{n-2}{k-2} + n^2  \ll O(n^k) \, ,
\]
contrariwise to the fact that 
$|\mD| = \binom{n}{k} = \Omega(n^k)$. 
Similarly, letting 
\[
\mF_{\mD}^{(k)} = \left\{B\in \binom{[n]}{2} : m(B;\mD) = \Omega(n^{k-2})\right\} \, ,
\]
it holds  $|\mF_{\mD}^{(k)}|=\Omega(n^2)$.

\begin{proof}[Proof of Theorem~\ref{thm:simple0}]
To simplify notation, set $\mH_n = \binom{[n]}{k}$, for $n\ge 4$, and let $G_n\in\mG(n,p_n\, ; \,\mH_n)$. 
\begin{enumerate}
\item The result follows from Lemma~\ref{all_edges} upon observing that 
$\Prb(G_n \text{ is empty}) \le\Prb(G_n \text{ is simple})$.  
\item Assume now that $p_n\gg \frac{1}{n^{k-1}}$. 
For each $\mD\in\mathscr{S}(\mH_n)$, and  each $B\in \binom{[n]}{2}$,  
let $I_B(\mD)$ denote the indicator of the event ``the coin-flips corresponding to $\mD$ resulted in at least two edges between the vertices of $B$". 

For each $\mD\in\mathscr{S}(\mH_n)$, define the random variable $Y_{\mD} = \sum_{B\in \mF_{\mD}} I_B(\mD)$. Then $Y_{\mD}$ counts the number of pairs of vertices from $[n]$ for which the 
coin-flips corresponding to $\mD$ resulted in at least two edges between them. Hence 
\[
\Prb(G_n\, \text{is simple}) = \binom{k}{2}^{-|\mH_n|} \sum_{\mD\in\mathscr{S}(\mH_n) }\Prb(Y_{\mD} =0) \, . 
\]
We aim to show that $\Exp(Y_{\mD})\to \infty$, for all $\mD\in\mathscr{S}(\mH_n)$. 
For every  $\mD\in\mathscr{S}(\mH_n)$  and  each $B\in\mF_{\mD}$ it holds 
\[
\Exp(I_B(\mD))=\Prb(I_B(\mD) = 1) = \Prb(\Bin(m(B;\mD), p_n)\ge 2) \, ,
\]
thus
%and therefore the fact that  $|\mF_{\mD}|= \Omega(n^2)$ implies that 
\begin{equation}\label{mean_Y}
\Exp(Y_{\mD}) = \sum_{B\in \mF_{\mD}} \Prb(\Bin(m(B;\mD), p_n)\ge 2) \ge \sum_{B\in \mF_{\mD}^{(k)}} \Prb(\Bin(m(B;\mD), p_n)\ge 2) \, .  
\end{equation}

We  distinguish two cases. Assume first that $p_n\ge \frac{1}{n^{k-2}}$. 
Then the case $\varepsilon =1$ in Lemma~\ref{bin_ineq} implies that, for large enough $n$ and $B\in \mF_{\mD}^{(k)}$, it holds 
\[
\Prb(\Bin(m(B;\mD), p_n)\ge 2) \ge \Prb\left(\Bin\left(m(B;\mD),  \frac{1}{n^{k-2}}\right)\ge 2\right) \ge \frac{1}{n}\, .
\]
Hence~\eqref{mean_Y} yields 
\[
\Exp(Y_{\mD}) \ge \sum_{B\in \mF_{\mD}^{(k)}} \frac{1}{n} = \Omega(n) \to \infty \, .
\]

Suppose now that $p_n\gg \frac{1}{n^{k-1}}$ but $p_n< \frac{1}{n^{k-2}}$. Then there exists $\varepsilon\in (0,1)$ such that $p_n \ge \frac{1}{n^{k-1-\varepsilon}}$, for $n$ large enough.  
For such $n$ and $B\in \mF_{\mD}^{(k)}$, Lemma~\ref{bin_ineq} implies that 
\[
\Prb(\Bin(m(B;\mD), p_n)\ge 2) \ge \Prb\left(\Bin\left(m(B;\mD), \frac{1}{n^{k-1-\varepsilon}}\right)\ge 2\right) \ge \frac{1}{n^{2-\varepsilon}} \, ,
\]
and~\eqref{mean_Y} now yields 
\[
\Exp(Y_{\mD}) \ge \sum_{B\in \mF_{\mD}^{(k)}} \frac{1}{n^{2-\varepsilon}} = \Omega(n^{\varepsilon}) \to \infty \, .
\]
Summarising the above, we have shown that $\Exp(Y_{\mD})$ tends to infinity when $p_n\gg\frac{1}{n^{k-1}}$. We now look at the variance of $Y_{\mD}$. 
Observe that, for each $\mD\in\mathscr{S}(\mH_n)$, the indicators 
$\{I_B(\mD)\}_{B\in\mF_{\mD}}$ are mutually independent; thus    
\begin{eqnarray*}
\text{Var}(Y_{\mD}) &=& \sum_{B\in\mF_{\mD}} \text{Var}(I_B(\mD)) \\
&=& \sum_{B\in\mF_{\mD}} \left(\Exp(I_B(\mD)) - (\Exp(I_B(\mD)))^2 \right)\\
&\le&  \sum_{B\in\mF_{\mD}} \Exp(I_B(\mD))\\ 
&=& \Exp(Y_{\mD}).
\end{eqnarray*}
The second moment method now yields  
\[
\Prb(Y_{\mD} = 0) \le \frac{\text{Var}(Y_{\mD})}{(\Exp(Y_{\mD}))^2} \le \frac{1}{\Exp(Y_{\mD})}  \to 0 , \, \text{ as } n\to \infty\,. 
\]
Summarizing the above, we have shown that 
$\Prb(\mG(n,p_n\, ; \,\mD) \, \text{is simple})\to 0$, for all $\mD\in\mathscr{S}(\mH_n)$, and so  
$\Prb(G_n\,\text{is simple})\to 0$. The result follows. 
\end{enumerate}
\end{proof}

\subsection{Proof of Theorem~\ref{thm:connected}}

The proof is analogous to the proof of the connectivity threshold in $\mG(n,p)$. 
We begin with determining a threshold for not having isolated vertices. 

\begin{theorem}\label{thm:isolated}
The property ``Not having an isolated vertex" has threshold  $\frac{\log(n)}{n^2 }$ in $\mG(n,p\, ; \,\binom{[n]}{3})$.
\end{theorem}
\begin{proof}
For each $n\ge1$, let $G_n \in \mG(n,p_n\, ; \, \binom{[n]}{3})$. 
For $i\in [n]$, let $I_i$ denote the indicator of the event ``vertex $i$ is isolated". 
Then $X_n=\sum_{i\in [n]}I_i$ is the number of isolated vertices in $G_n$. 
From Corollary~\ref{cor:1} it follows that 
\[
\mu_n := \Exp(X_n) = n\cdot \left(1 - \frac{2p_n}{3} \right)^{\binom{n-1}{2}} \, .
\]
Therefore, using the inequality $1-x\le e^{-x}$, for $x\in \mathbb{R}$, we conclude that  
\begin{eqnarray*}
\Exp(X_n) &\le&  e^{\log(n) - \frac{2p_n}{3} \cdot \binom{n-1}{2}} \to 0,\,  \, \text{ when } \, p_n \gg \frac{\log(n)}{n^2} \, .
\end{eqnarray*}
Markov's inequality now gives 
$\Prb(X_n \ge 1) \le \Exp(X_n) \to 0$, 
and therefore $\Prb(X_n = 0) \to 1$. 

Now assume that $p_n \ll \frac{\log(n)}{n^2 }$. 
For $i\in [n]$, let $A_i$ be the event  that ``vertex $i$ is isolated". Then, using the inequality 
$1-x \ge e^{-2x}$, for $x\in [0,1/2]$, we conclude, for large enough values of $n$, that  
\[
\mu_n = \sum_{i\in [n]} \Prb(A_i) = n\cdot \left(1 - \frac{2p_n}{3} \right)^{\binom{n-1}{2}} \ge 
e^{\log(n) - \frac{4p_n}{3} \cdot \binom{n-1}{2}} \to +\infty \, .
\]
Now, for any two distinct vertices 
$i,j\in [n]$, let $\mH_i = \{H\in \binom{[n]}{3} : i\in H \, , \, j \notin H\}$, 
$\mH_j= \{H\in \binom{[n]}{3} : j\in H \, , \, i \notin H\}$ and $\mH_{ij} =  \{H\in \binom{[n]}{3} : \{i,j\}\subset H\}$. Then it holds $|\mH_i| = |\mH_j| = \binom{n-2}{2}$ and $|\mH_{ij}| = n-2$, and so 
\[
\Prb(A_i \cap A_j) = \left(1-\frac{2p_n}{3}\right)^{2\binom{n-2}{2}} \cdot (1-p_n)^{n-2} \le  \left(1-\frac{2p_n}{3}\right)^{2\binom{n-2}{2} + (n-2)} . 
\]
We conclude that 
\[
\gamma_n := \sum_{i\neq j} \Prb(A_i\cap A_j) \le n(n-1) \left(1-\frac{2p_n}{3}\right)^{2\binom{n-2}{2} + (n-2)}   \, ,
\]
which yields   
\begin{eqnarray*}
\frac{\Exp(X_n^2)}{(\Exp(X_n))^2} &=& \frac{\mu_n + \gamma_n}{\mu_n^2} \\
&=& \frac{1}{\mu_n} + \frac{\gamma_n}{\mu_n^2} \\
&\le& \frac{1}{\mu_n} +   \frac{     n(n-1) \left(1-\frac{2p_n}{3}\right)^{2\binom{n-2}{2} + (n-2)}    }{n^2\cdot \left(1 - \frac{2p_n}{3} \right)^{2\binom{n-1}{2}}} \\ 
&\le&  \frac{1}{\mu_n} + \frac{1}{\left(1-\frac{2p_n}{3}\right)^{ n-2 }  } \\
&\le&  \frac{1}{\mu_n} + \frac{1}{\left(1-p_n\right)^{ n }  }   \, .
\end{eqnarray*}
Now notice that the assumption $p_n \ll \frac{\log(n)}{ n^2 }$  implies 
\[
\left(1-p_n\right)^{ n } \ge e^{-2np_n} \to 1 \, .
\]
Since $\mu_n\to +\infty$, the second moment method  yields 
\[
\Prb(X_n \ge 1) \ge \frac{(\Exp(X_n))^2}{\Exp(X_n^2)}  \to 1 \, . 
\]
This implies that $\Prb(X_n=0) \to 0$, and the result follows. 
\end{proof}

We may now proceed with the proof of Theorem~\ref{thm:connected}.

\begin{proof}[Proof of Theorem~\ref{thm:connected}]
Suppose first that $p_n \ll \frac{\log(n)}{n^2}$. Then, Theorem~\ref{thm:isolated} implies that the probability that $G_n\in\mG(n,p_n\, ; \,\binom{[n]}{3})$ has isolated vertices tends to one, as $n$ tends to infinity; hence the same holds true for the probability that $G_n\in\mG(n,p_n\, ; \,\binom{[n]}{3})$ is disconnected. 

Suppose now that $p_n\gg \frac{\log(n)}{n^2}$. Let  
$P_n$ be the probability of the event that ``a graph $G_n\in \mG(n,p_n\, ; \,\binom{[n]}{3})$ is disconnected". We aim to show that $P_n$ tends to zero.  
For each $m\le n/2$, let $X_m$ be the number of subsets of $m$ vertices that are disconnected from all remaining vertices of the graph $G_n$. Then
\[
P_n \le \Prb\left(\sum_{m=1}^{n/2} X_m > 0\right) \le  \sum_{m=1}^{n/2} \Exp(X_m) \, , 
\]
and it is therefore enough to bound $\Exp(X_m)$. Now observe that 
\[
\Exp(X_m) = \binom{n}{m} \cdot \left(1-\frac{2p_n}{3}\right)^{(n-m)\cdot \binom{m}{2}} \cdot \left(1-\frac{2p_n}{3}\right)^{m\cdot \binom{n-m}{2}} \, ,
\]
since, for each subset of $[n]$ consisting of $m$ vertices, say $M$, there are $(n-m)\binom{m}{2}$ sets $H\in\binom{[n]}{3}$ having exactly two vertices in $M$, and $m\cdot \binom{n-m}{2}$ sets $H\in\binom{[n]}{3}$ having exactly two vertices in $[n]\setminus M$, and for each of the aforementioned  $H\in\binom{[n]}{3}$ the probability that no edge is introduced  between a vertex in $M$ and a vertex in $[n]\setminus M$ is equal to $1-\frac{2p_n}{3}$. Hence, using the  estimates $m\le n/2$ and 
$m! \ge \left(\frac{m}{e}\right)^m$, 
we conclude that 
\begin{eqnarray*}
\Exp(X_m) &=&  \binom{n}{m} \cdot \left(1-\frac{2p_n}{3}\right)^{(n-m)\cdot \binom{m}{2} +m\cdot \binom{n-m}{2} } 
= \binom{n}{m} \cdot  \left(1-\frac{2p_n}{3}\right)^{ \frac{m(n-m)(n-2)}{2}} \\ 
&\le& \frac{n^m}{m!} \cdot \left(1-\frac{2p_n}{3}\right)^{ \frac{m n (n-2)}{4}} \le e^m \cdot n^m\cdot 
\left(1-\frac{2p_n}{3}\right)^{ \frac{m n (n-2)}{4}} \\
&=&\left( e\cdot n \cdot \left(1-\frac{2p_n}{3}\right)^{\frac{n(n-2)}{4}}     \right)^m \, .
\end{eqnarray*}
We therefore have 
\[
\sum_{m=1}^{n/2} \Exp(X_m) \le \sum_{m\ge 1} \left( e\cdot n \cdot \left(1-\frac{2p_n}{3}\right)^{\frac{n(n-2)}{4}}     \right)^m \, .
\]
Now observe that, for $p_n\gg \frac{\log(n)}{n^2}$, it holds  
\[
A_n:=e\cdot n \cdot \left(1-\frac{2p_n}{3}\right)^{\frac{n(n-2)}{4}} \le e\cdot n \cdot e^{-\frac{2p_n}{3} \cdot \frac{n(n-2)}{4}} = e\cdot e^{\log(n)  -\frac{2p_n}{3} \cdot \frac{n(n-2)}{4}} \to 0 \, ,
\]
and hence, for $n$ large enough,  we have  
\[
\sum_{m=1}^{n/2} \Exp(X_m) \le \frac{A_n}{1-A_n} \to 0 \, .
\]
The result follows. 
\end{proof}

\section{Triangles}\label{sec:triangles}

\subsection{Proof of Theorem~\ref{thm:triangle_rand}}

Let $G_n\in\mG(n,p\, ; \,\mbH_{n,q;3} )$.  Let also $X$ count the number of triangles in $G_n$, and for each potential triangle $T\in\binom{[n]}{3}$, let $I_T$ be the indicator of the event ``the triangle $T$ is present in $G_n$". 
Then 
$\Exp(X) = \sum_{T\in\binom{[n]}{3}}\Prb(I_T =1)$, and the linearity of expectation implies that it is enough 
to compute $\Prb(I_T =1)$ for a particular $T\in\binom{[n]}{3}$.
 
Let $T\in\binom{[n]}{3}$ be fixed. 
Recall that for every $H\in\mbH_{n,q;3}$ we choose a doubleton $D\in\binom{H}{3}$ uniformly at random, and then flip a coin to decide whether we join the two vertices vertices in $D$ with an edge. 
In particular, we may flip such a coin for $H=T$; this happens with probability $q$. 
So assume first that $T\in\mbH_{n,q;3}$ and let $\binom{T}{2}= \{D_1,D_2,D_3\}$. 
For every $i=1,2,3$, let $X_i$ be the number of hyperedges $H\in\mbH_{n,q;3}\setminus \{T\}$ such that $D_i\subset H$, and observe that $X_i\sim\Bin(n-3,q)$, because there are $n-3$ elements in $\binom{[n]}{3}$ that contain $D$ and each such element is included in $\mbH_{n,q;3}$ with probability $q$. 
Note that the random variables 
$X_{i}, i=1,2,3,$ are  independent and identically distributed. Let $D$ be the doubleton which is chosen from $T$ uniformly at random. 
Then, given that $T\in \mbH_{n,q;3}$, we have 
\begin{equation}\label{ddd}
\Prb(I_T =1) = \sum_{i=1}^{3} \Prb(I_T =1 \, |\, D=D_i) \cdot \Prb(D=D_i) = \frac{1}{3}\sum_{i=1}^{3} \Prb(I_T =1 \, |\, D=D_i) \, .
\end{equation}
Now observe that it holds 
\[
\Exp(I_T  \, |\, D=D_1, X_1,X_2,X_3) = (1-p)\cdot \prod_{i=1}^{3} \Prb(\Bin(X_{i},\frac{p}{3})\ge 1)  + p \cdot \prod_{j\neq 1} \Prb(\Bin(X_{j},\frac{p}{3})\ge 1) \, ,
\]
because, given $X_1,X_2,X_3$, if the outcome of the coin-flip corresponding to $D_1$ is a failure, then we introduce no edge between the two vertices of $D_1$ and therefore the probability that the two vertices of every  $D\in\binom{T}{3}$ are joined by at least one edge is equal $\prod_{i=1}^{3} \Prb(\Bin(X_{i},\frac{p}{3})\ge 1)$. If the 
outcome of the coin-flip corresponding to $D_1$ is a success, then one edge is introduced between the two vertices of $D_1$ and the probability that the two vertices of every $D\in \binom{T}{3}\setminus \{D_1\}$ are joined by at least one edge is then equal to $\Prb(\Bin(X_{2},\frac{p}{3})\ge 1)\cdot \Prb(\Bin(X_{3},\frac{p}{3})\ge 1)$. 
Now note that, since $X_i\sim\Bin(n-3,q)$, for every $i=1,2,3$, it holds 
$\Bin(X_{i},\frac{p}{3}) \sim \Bin(n-3,\frac{pq}{3})$ and therefore it holds 
$\Exp\left( \Prb(\Bin(X_{i},\frac{p}{3})\ge 1)\right) = \Prb(\Bin(n-3,\frac{pq}{3})\ge 1)$; this together with the independence of the random variables $X_1,X_2,X_3$ imply that 
\begin{eqnarray*}
\Prb(I_T =1 \, |\, D=D_1)&=& \Exp\left(\Exp(I_T  \, |\, D=D_1, X_1,X_2,X_3)\right) \\
&=& (1-p)\cdot \left(\Prb(\Bin(n-3,\frac{pq}{3})\ge 1)\right)^3 + p\cdot \left(\Prb(\Bin(n-3,\frac{pq}{3})\ge 1)\right)^2 \, .
\end{eqnarray*}
By symmetry, it holds 
\[
\Prb(I_T =1 \, |\, D=D_1)=\Prb(I_T =1 \, |\, D=D_2)=\Prb(I_T =1 \, |\, D=D_3) \, ,
\]
and we conclude from~\eqref{ddd} that 
\begin{equation}\label{ddq}
\Prb(I_T =1\,|\, T\in\mbH_{n,q;3}) = (1-p)\cdot \left(\Prb(\Bin(n-3,\frac{pq}{3})\ge 1)\right)^3 + p\cdot \left(\Prb(\Bin(n-3,\frac{pq}{3})\ge 1)\right)^2 \, .
\end{equation}
Suppose now that $T\notin\mbH_{n,q;3}$. Then, using an analogous argument as before, we have 
\begin{equation}\label{dd1-q}
\Prb(I_T =1\,|\, T\notin\mbH_{n,q;3}) =  \left(\Prb(\Bin(n-3,\frac{pq}{3})\ge 1)\right)^3 \, .
\end{equation}
Hence~\eqref{ddq} and~\eqref{dd1-q} imply that  
\begin{eqnarray*}
\Prb(I_T =1) &=& q\cdot \Prb(I_T =1\,|\, T\in\mbH_{n,q;3})  + (1-q) \cdot \Prb(I_T =1\,|\, T\notin\mbH_{n,q;3}) \\
&=& (1-pq)\cdot \left(\Prb(\Bin(n-3,\frac{pq}{3})\ge 1)\right)^3 + pq\cdot \left(\Prb(\Bin(n-3,\frac{pq}{3})\ge 1)\right)^2 \, ,
\end{eqnarray*}
and Theorem~\ref{thm:triangle_rand} follows.

\subsection{Proof of Theorem~\ref{thm:triangle_rand2}}

The proof is similar to the proof of Theorem~\ref{thm:triangle_rand}, so we sketch it. 

Let $G_n\in\mG(n,p\,;\, \mbH_{n,m;3})$, let $T\in\binom{[n]}{3}$ be fixed, and let $I_T$ be the  indicator of the event ``the triangle $T$ is present in $G_n$".
We consider two cases. 
Either $T\in \mbH_{n,m;3}$ or $T\notin \mbH_{n,m;3}$. 

Suppose first that $T\in \mbH_{n,m;3}$; this happens with probability $\frac{m}{\binom{n}{3}}$. 
Suppose that $\binom{T}{2} = \{D_1,D_2,D_3\}$ and,
for $i=1,2,3$, let $W_i$ be the number of hyperedges in $\mbH_{n,m;3}\setminus \{T\}$  that contain $D_i$. 
Observe that $W_i \sim \Hyp(\binom{n-3}{3}, n-3, m)$, for each $i=1,2,3$, and that the 
random variables $W_1,W_2,W_3$ are  independent and identically distributed. 

Upon conditioning on the outcome of the coin-flip corresponding to $T$ and on $W_1,W_2,W_3$, we deduce that  
\[
\Prb(I_T=1\, |\, T\in \mbH_{n,m;3}) = (1-p) \cdot \left( \Prb(\Bin(W,p)\ge 1) \right)^3 + p\cdot \left( \Prb(\Bin(W,p)\ge 1) \right)^2 \, , 
\]
where $W\sim \Hyp(\binom{n-3}{3}, n-3, m)$. 

Similarly, we find 
\[
\Prb(I_T=1\, |\, T\notin \mbH_{n,m;3}) =  \left( \Prb(\Bin(W,p)\ge 1) \right)^3 \, , 
\]
where $W\sim\Hyp(\binom{n-3}{3}, n-3, m)$, 
and Theorem~\ref{thm:triangle_rand2} follows.

\subsection{Proof of Theorem~\ref{thm:triangle_four}}

Let $\mH_n = \binom{[n]}{4}$, and let $G_n\in\mG(n,p\, ; \,\mH_n)$. 
We begin with finding the probability, say $\alpha$, that all pairs of vertices of a particular triplet $T\in\binom{[n]}{3}$ are joined by at least one edge. 
Let $\mH_{T} = \{H\in\mH_n : T\subset H\}$, and note that $|\mH_T|=n-3$. 
Also note that any element $H\in\mH_n\setminus\mH_T$ satisfies $|H\cap T|\le 2$. 
We  compute the desired probability via conditioning on the number of pairs of vertices of $T$ which have been joined by at least one edge after ``flipping the 
coins corresponding to $\mH_T$".

We  begin with choosing a doubleton uniformly at random from every element of $\mH_T$, and then flip a coin for every such doubleton in order to decide whether we put an edge between the corresponding edges. 
Let $Y$ count the number of doubletons $D\in \binom{T}{2}$ whose elements have been joined with at least one edge in the graph. 
Observe that $Y\in\{0,1,2,3\}$, and let $\pi_i = \Prb(Y=i)$, for $i=0,1,2,3$. 

Now, given that $Y=0$, the probability that $T$ is a triangle in $G_n$ equals  
\[
\alpha_0 = \left\{\Prb\left(\Bin\left(\binom{n-3}{2}, \frac{p}{6} \right)\ge 1 \right)\right\}^3 \, , 
\]
since for every $D\in\binom{T}{2}$ there are $\binom{n-3}{2}$ elements $H\in\mH\setminus\mH_T$ 
that contain $D$, and the probability that the coin-flip corresponding to a particular $H\in\mH\setminus\mH_T$ which contains $D$ results in an edge between the elements of $D$ is equal to $\frac{p}{6}$. 
Similarly, given that $Y=i$, for $i=1,2$,  the probability that $T$ is a triangle in $G_n$ equals
\[
\alpha_i =  \left\{\Prb\left(\Bin\left(\binom{n-3}{2}, \frac{p}{6} \right)\ge 1 \right)\right\}^{3-i} \, .
\]
Hence the desired probability is equal to 
\[
\alpha = \sum_{i=0}^3 \alpha_i \cdot \Prb(Y=i) \, ,
\]
and it remains to determine the probability distribution of $Y$. 
Consider the absorbing homogeneous Markov chain on the state space $S=\{0,1,2,3\}$ having transition matrix given by 
\[
P = 
\begin{blockarray}{ccccc}
 & 0 & 1 & 2 & 3 \\
\begin{block}{c[cccc]}
  0 & 1-p/2 & p/2 & 0 & 0  \\
  1 & 0    & 1-p/3    & p/3 & 0  \\
  2 & 0 & 0 & 1-p/6 & p/6  \\
  3 & 0 & 0 & 0 & 1  \\ 
\end{block} 
\end{blockarray} 
 \]
We claim that for $i=0,1,2$, $\Prb(Y=i)$ is equal to the $(0,i)$-element of $P^{n-3}$, denoted $p_{0i}^{(n-3)}$. 
In other words, we claim that $\Prb(Y=i)$ is equal to the probability that, beginning from state $0$, a Markov chain with transition matrix $P$ will be in state $i$ after $n-3$ steps. 
Given the claim, we then also obtain $\Prb(Y=3) = 1 - \sum_{i=0}^{2}p_{0i}^{(n-3)}$.

To prove the claim, recall that a sample from $Y$ is obtained after ``flipping the 
coins corresponding to $\mH_T$". We now flip the coins corresponding to $\mH_T$ one by one in a series of steps. 
Let $H_1,\ldots, H_{n-3}$ be an enumeration of the elements of $\mH_T$. 
For each step $t=1,\ldots, n-3$, we choose  a doubleton $D_{t}\in \binom{H_i}{2}$ uniformly at random and then flip a coin in order to decide whether we join the elements of $D_i$ with an edge in the graph.  
For each step $t=1,\ldots, n-3$, let $Y^{(t)}$ be the number of elements in $\binom{T}{2}$ which have been 
joined with at least one edge at step $t$, and set $Y^{(0)} = 0$. 
Note that $Y^{(n-3)}=Y$ and that if $Y^{(t-1)}=j$, for some $j=0,1,2$, then $Y^{(t)}\in \{j,j+1\}$. 
Now observe that for any $t\ge 1$ and any $j=0,1,2$ we have
\[
\Prb(Y^{(t)} = j+1 \, |\, Y^{(t-1)}=j) = \frac{3-j}{6}\cdot p \, ,
\]
because, given that $Y^{(t-1)}=j$, exactly  $3-j$ out of the six doubletons in $\binom{H_{t}}{2}$ are favourable (in the sense that they may potentially  introduce an edge between two previously unjoined vertices of $T$) 
and, given that a favourable doubleton is selected, there is a probability $p$ that an edge will be introduced between the 
two corresponding vertices of $T$. 
In other words, $\{Y^{(t)}\}_{t\ge 0}$ is a homogeneous absorbing Markov chain, which begins at state $0$, having transition matrix given by $P$. Theorem~\ref{thm:triangle_four} follows upon recalling that $Y= Y^{(n-3)}$.

\end{document}